\documentclass{psp}

\usepackage{eucal,url,amssymb, 
enumerate,amscd,
}
\usepackage[pagebackref]{hyperref}  

\newtheorem{thm}{Theorem}[section]

\newtheorem{lem}[thm]{Lemma}

\newtheorem{prop}[thm]{Proposition}

\newtheorem{cor}[thm]{Corollary}

\newtheorem{dfn}[thm]{Definition}
\newcommand{\f}[1]{\ensuremath{\mathfrak{#1}}}
\newcommand{\core}[1]{\ensuremath{{\rm{core}}(#1)}}

\newtheorem{rmrk}[thm]{Remark}

\newtheorem{exm}[thm]{Example}

\newcommand{\m}{\frak{m}}
\newcommand{\ul}[1]{\ensuremath{\underline{#1}}}
\newcommand{\ol}[1]{\ensuremath{\overline{#1}}}

\newcommand{\depth}{{\rm depth}}

\def\sqr#1#2{{\vcenter{\hrule height.#2pt
        \hbox{\vrule width.#2pt height#1pt \kern#1pt
                \vrule width.#2pt}
        \hrule height.#2pt}}}

\begin{document}

\title{The $cl$-$\rm core$ of an ideal}

\author[L. Fouli and J. C. Vassilev] 
{Louiza Fouli\\ 
Department of Mathematical Sciences,New Mexico State University \\
Las Cruces, NM 88003-8001, USA\\
e-mail\textup{: \texttt{lfouli@math.nmsu.edu}}
\and\ Janet C. Vassilev\\ 
Department of Mathematics, University of New Mexico\\
Albuquerque, NM  87131-0001, USA \\
e-mail\textup{: \texttt{jvassil@math.unm.edu}}} 





\maketitle

\begin{abstract}{We expand the notion of core to $cl$-core for Nakayama closures $cl$.
In the characteristic $p>0$ setting, when $cl$ is the tight closure, denoted by $*$, we
give some examples of ideals when the core and the $*$-core differ.
We note that  $*$-core$(I)=$ core$(I)$, if $I$ is an ideal in
a one-dimensional domain with infinite residue field or if $I$ is an ideal
generated by a system of parameters in any Noetherian ring. More
generally, we show the same result in a  Cohen--Macaulay normal local
domain with infinite perfect residue field, if the analytic spread,
$\ell$, is equal to the $*$-spread and $I$ is $G_{\ell}$ and
weakly-$(\ell-1)$-residually $S_2$. This last is dependent on our result
that generalizes the notion of general minimal reductions to general minimal
$*$-reductions.  We also determine that the $*$-core of a tightly closed
ideal in certain one-dimensional semigroup rings is tightly closed and
therefore integrally closed.}
\end{abstract}

\section{Introduction}
The core of an ideal, the intersection of all reductions of the ideal, was
introduced by Rees and Sally in \cite{RS} in the 80's. Then over a decade
passed before Huneke and Swanson \cite{HS1} analyzed the core of ideals in
2-dimensional regular local rings. Then a stream of papers came out within
a decade by Corso, Polini and Ulrich \cite{CPU1}, \cite{CPU2}, \cite{PU},
Hyry and Smith \cite{HyS1}, \cite{HyS2} and Huneke and Trung \cite{HT}
expanding the understanding and computability of the core. As it is the
intersection of reductions, in general it lies deep within the ideal.  In
fact, the core is related to the {\it Brian\c{c}on-Skoda} Theorem
\cite{LS}: Let $R$ be a regular local ring of dimension $d$ and let $I$ an ideal. Then $\overline{I^d}
\subseteq J$ for any reduction $J$ of $I$.  Hence $\overline{I^d}
\subseteq \core{I}$.  A very slick proof of the Brian\c{c}on-Skoda
Theorem was given in characteristic $p>0$, using tight closure, \cite[Theorem~5.4]{HH}.   We expand the notion of core to other closure operations; in
particular, Nakayama closure operations. Epstein defined the 
Nakayama closure as follows:

\begin{dfn}$($\cite{Ep}$)${\rm
A closure operation $cl$, defined on a Noetherian local ring
$(R,\frak{m})$ is a Nakayama closure if for all ideals $I$ and $J$ satisfying $J \subset
I \subset (J+\frak{m}I)^{cl}$ it follows that $I \subset J^{cl}$.}
\end{dfn}

Note that integral closure, tight closure and Frobenius closure are
examples of Nakayama closures, \cite[Proposition~2.1]{Ep}. Recall that
both the tight closure and the Frobenius closure are characteristic $p>0$
notions.

Epstein's main reason for the definition of Nakayama closure was to
expand the notion of reduction and analytic spread to these other closure
operations. With a well defined notion of reduction and analytic spread, we
can easily extend the definition of the core to these other closure
operations. In general, the $cl$-cores will not lie as deep in the
ideal as the core itself. This will follow from the fact that the
partial ordering of closure operations leads to a reverse partial
ordering on the $cl$-cores (Proposition~\ref{core inclu}).  Our hope in studying these $cl$-cores
is that tight closure methods may be used to compute the core in
situations where the core and the $*$-core agree.

In Section 2, we provide some background information about the core and
tight closure theory, along with a review of some central theorems that
are used in this article.  In Section 3, we review $cl$-reductions of
ideals. We also discuss the $cl$-spread of an ideal and define both the
$cl$-deviation and the second $cl$-deviation in terms of the $cl$-spread.
We also introduce the notion of $cl$-core.  In Section 4, we show
different instances when the core and the $*$-core agree. Our main result,
Theorem~\ref{extension of main thm}, shows that we can form general minimal
$*$-reductions. This allows us to show in particular that if $(R, \f{m})$
is a Gorenstein normal local isolated singularity  of positive characteristic with infinite
perfect residue field, test ideal equal to $\f{m}$ and $I$ is an
$\f{m}$-primary tightly closed ideal then $*$-$\core{I}=\core{I}$ (Corollary~\ref{iso sing cores equal}).  Also,
when $(R,\f{m})$ is a Cohen--Macaulay normal domain of positive characteristic and infinite perfect residue field and $I$ is an ideal that satisfies
$G_{\ell}$ and is weakly $(\ell-1)$-residually $S_2$ with
$\ell^*(I)=\ell(I)=\ell$ then $\core{I}=*$-$\core{I}$ (Theorem~\ref{core agree}). In Section 5, we
discuss when the $*$-core is tightly closed in some one-dimensional
semigroup rings.  Finally, in Section 6, we give some examples where we compute the $*$-core and in each case we compare the core with the $*$-core.

\section{Background}

In this section we recall some definitions and results that we will use extensively
in this article.

\begin{dfn}{\rm
Let $R$ be a Noetherian local ring of characteristic $p>0$. We denote positive powers of $p$ by
$q$ and the set of elements of $R$ which are not contained in the union of
minimal primes by $R^0$. Then
\begin{enumerate}[(a)]
\item \ For any ideal $I \subset R$, $I^{[q]}$ is
the ideal generated by the $q$th powers of elements in $I$.

\item  \ We say
an element $x \in R$ is in the {\it tight closure}, $I^{*}$, of $I$ if there
exists a $c \in R^0$, such that $cx^q \in I^{[q]}$ for all large
$q$.

\item  \ We say an element $x \in R$ is in the {\it Frobenius closure}, $I^F$,
of $I$ if $x^q \in I^{[q]}$ for all large $q$.
\end{enumerate}}
\end{dfn}

Finding the tight closure of an ideal
would be hard without test elements and test ideals.  A {\it test
element} is an element $c \in  R^{0}$
such that $c I^* \subset I$ for all $I \subset R$. We note here that $c \in
\bigcap\limits_{I \subset R} (I:I^*)$.  Since the intersection of
ideals is an ideal we call the ideal $\tau=\bigcap\limits_{I
\subset R} (I:I^*)$ {\it the test ideal} of $R$, namely the  ideal generated by
all the test elements.
We say that $I$ is a {\it parameter ideal} if $I$ is generated by part of
a system of parameters.  In a Gorenstein local isolated singularity, the
following theorem of Smith \cite{Sm} gives a nice way to compute the
tight closure of a parameter ideal using the test ideal.

\begin{thm}{\rm(\cite[Lemma~3.6, Proposition~4.5]{Sm})} \label{tcsop}Let $R$ be a
Gorenstein local isolated singularity with test ideal $\tau$.  Then for
any system of parameters $x_1,x_2, \ldots, x_d$,
$$(x_1,x_2,\ldots,x_d):\tau=(x_1,x_2,\ldots,x_d)^*.$$
\end{thm}

Related concepts are parameter test elements and parameter test
ideals. A {\it parameter test element} is an element $c \in R^{0}$  such that  $c I^* \subset I$ for all
parameter ideals $I \subset R$. Note that $c \in \bigcap\limits_{I
\subset R} (I:I^*)$.   Let $P(R)$ be the set of parameter ideals
in $R$.
 We call $\tau_{par}=\bigcap\limits_{I \in P(R)} (I:I^*)$ {\it the parameter test
ideal}.  It is known in a Gorenstein ring that $\tau=\tau_{par}$.  We can
relax the Gorenstein assumption from the above theorem and obtain:

\begin{thm}{\rm (\cite{Va1})}{\label{sop colons}} Let $R$ be a Cohen--Macaulay local  isolated singularity with
parameter test ideal $\tau_{par}$.  For any system of
parameters $x_1,x_2, \ldots, x_d$,
$$(x_1,x_2,\ldots,x_d):\tau_{par}=(x_1,x_2,\ldots,x_d)^*.$$
\end{thm}

Note, if the test ideal is known to be $\f{m}$, where $\m$ is the maximal ideal of the ring,  then the parameter  test ideal will also be $\f{m}$.

Another result that we will use repeatedly is the following due to Aberbach:

\begin{prop} {\rm (\cite[Proposition~2.4]{Ab})}\label{Ab}
Let $(R, \f{m})$ be an excellent, analytically irreducible local ring of characteristic $p >0$, let $I$ be an
ideal, and let $f\in R$. Assume that $f\not \in I^*$. Then there exists $q_0=p^{e_0}$ such that for all
$q \geq q_0$ we have $I^{[q]}:f^q \subset \f{m}^{[q/q_0]}$.
\end{prop}

Notice that later on we will be assuming that $R$ is an excellent normal local domain,
which implies that $R$ is analytically irreducible, since the completion of an excellent
normal domain is again a normal domain. Hence one may use Proposition~\ref{Ab}.

Let $R$ be a Noetherian ring and $I$ an ideal. We say that $J
\subset I$ is a {\it reduction} of $I$ if $I^{n+1}=JI^{n}$ for some
nonnegative integer $n$. Northcott and Rees introduced this notion
in \cite{NR} in order to study multiplicities. If $(R,\m)$ is a Noetherian local ring and $I$ is an $\m$-primary ideal then $I$ and its
reduction $J$ have the same multiplicity and thus one may want to
shift the attention from $I$ to the simpler ideal $J$. If $R$ is
a  Noetherian local ring with infinite residue field then $I$ has infinitely many minimal reductions or $I$ is basic, i.e. $I$ is the only reduction of itself (\cite{NR}). When $R$ is a Noetherian local ring and $I$ is an ideal then a reduction $J$
of $I$ is called minimal if it is minimal with respect to inclusion.
To facilitate the lack of uniqueness for minimal reductions, Rees
and Sally introduced the core of an ideal:

\begin{dfn}{\rm(\cite{RS})} \label{defcore} {\rm Let $R$ be a Noetherian ring. Let $I$ be an ideal.
Then \linebreak $\core{I} = \bigcap\limits_{J \subset I} \ J$, where $J$ is a reduction of $I$.}
\end{dfn}

When $R$ is a Noetherian local ring  it is enough to take the intersection over all
minimal reductions since every reduction contains a minimal
reduction. There has been a significant effort by several authors to
find efficient ways of computing this infinite intersection. One
result in particular is of special interest to us.

\begin{thm}$(${\rm \cite[Theorem~4.5]{CPU1}}$)$ \label{CPU gen red}
Let $R$ be a Cohen--Macaulay  local ring with infinite residue
field and $I$ an ideal of analytic spread $\ell$. Assume that $I$ satisfies $G_{\ell}$ and is weakly
$(\ell-1)$-residually $S_2$. Then $\core{I} = \f{a}_1\cap \ldots \cap\f{a}_t$ for
$\f{a}_1,\ldots,\f{a}_t$ general $\ell$-generated ideals in $I$ which are reductions of I and
for some finite integer $t$.
\end{thm}

We now explain the conditions in the statement of Theorem~\ref{CPU gen red}.

The {\it analytic spread} of $I$, $\ell(I)$, is the Krull dimension
of  $\mathcal{F}(I):=\bigoplus \limits_{i \geq
0} I^{i}/\f{m}I^{i}$, the special fiber ring of  $I$. It is well known that if $R$
is a Noetherian local ring with infinite residue field then any
minimal reduction $J$ of $I$ has the same minimal number of
generators, namely $\mu(J)=\ell(I)$ \cite{NR}. It is
straightforward to see that in general ${\rm ht} \; I \leq \ell(I)
\leq \dim R$.

Following the definitions given in \cite{CEU} we say that an ideal
$I$ satisfies the property $G_s$ if for every prime ideal
$\frak{p}$ containing $I$ with dim $R_{\frak{p}} \leq s-1$, the
minimal number of generators, $\mu(I_{\frak{p}})$, of $I_{\frak{p}}$
is at most dim $R_{\frak{p}}$. A proper ideal $K$ is called an {\it
$s$-residual intersection} of $I$ if there exists an $s$-generated
ideal $\frak{a} \subset I$ so that $K=\frak{a}:I$ and ${\rm ht} \; K \geq s \geq {\rm ht}\; I$.  If  ${\rm ht} \;I+K \geq s+1$, then $K$ is said to be a {\it geometric $s$-residual
intersection} of $I$. If $R/K$ is Cohen--Macaulay for every
$i$-residual intersection (geometric $i$-residual intersection) $K$
of $I$ and every $i \leq s$ then $I$ satisfies $AN_s$ ($AN_s^{-}$).
An ideal $I$ is called {\it $s$-residually $S_2$} ({\it weakly
$s$-residually $S_2$}) if $R/K$ satisfies Serre's condition $S_2$
for every $i$-residual intersection (geometric $i$-residual
intersection) $K$ of $I$ and every $i \leq s$.

\begin{rmrk}{\rm
Let $(R, \f{m})$ be a Noetherian local ring and $I$ an ideal.
Let $g={\rm ht} \; I$. It is not difficult for an ideal to satisfy the condition $G_s$. If $I$ is an $\f{m}$-primary ideal or in general an
equimultiple ideal, i.e. $\ell=\ell(I)={\rm ht}\; I$, then $I$
satisfies $G_{\ell}$ automatically.

If $(R,\m)$ is a Cohen--Macaulay local  ring of dimension $d$ and $I$ an ideal satisfying
$G_s$, then $I$ is universally $s$-residually $S_2$ in the following cases:
\begin{enumerate}[(a)]

\item \ $R$ is Gorenstein, and the local cohomology modules $H_{\f{m}}^{d-g-j}(R/I^j)$ vanish for
all $1\leq j \leq s-g +1$, or equivalently, $Ext_{R}^{g+j}(R/I^j,R) = 0$ for all $1 \leq j \leq s-g +1$
(\cite[Theorem~4.1 and 4.3]{CEU}).

\item  \ $R$ is Gorenstein, $\depth \;R/I^j \geq \dim \;R/I-j+1$ for all  $1 \leq j \leq  s-g +1$
(\cite[Theorem~2.9(a)]{U}).

\item \ $I$ has sliding depth (\cite[Theorem~3.3]{HVV}).

\end{enumerate}
Notice that condition $($b$)$ implies $($a$)$ and the property $AN_s$ by \cite[Theorem~2.9(a)]{U}.
Also the conditions $($b$)$ and $($c$)$ are satisfied by {\it strongly Cohen--Macaulay} ideals, i.e.
ideals whose Koszul homology modules are Cohen--Macaulay. If $I$ is a Cohen--Macaulay almost complete
intersection or a Cohen--Macaulay deviation two ideal of a Gorenstein ring then $I$
is strongly Cohen--Macaulay  (\cite[p. 259]{AH}). Furthermore, if $I$ is in the linkage class of a complete intersection (licci) then $I$ is again a strongly Cohen--Macaulay ideal (\cite[Theorem~1.11]{Hu1}). Standard examples include
perfect ideals of height two and perfect Gorenstein ideals of height three.}
\end{rmrk}

\section{$cl$-Reductions and the definition of $cl$-core}

Let $R$ be a Noetherian ring and $I$ an ideal. Recall that $J \subset I$ is a reduction of an ideal $I$ if
$JI^n=I^{n+1}$ for some nonnegative integer $n$. If $J$ is a reduction of $I$, then $J \subset I
\subset  \overline{J}$.  Epstein defines a $cl$-reduction of an
ideal $I$ to be an ideal $J$ such that $J \subset I \subset J^{cl}$.
If $cl$ is a Nakayama closure we have the following Lemma:

\begin{lem}$($\cite[Lemma 2.2]{Ep}$)$ \label{epred} 
Let $R$ be a Noetherian local ring and $I$ an ideal. If $cl$ is a Nakayama closure on $R$, then for
any $cl$-reduction $J$ of $I$, there is a  minimal $cl$-reduction $K$ of
$I$ contained in $J$. Moroever, in this situation any minimal generating
set of $K$ extends to a minimal generating set of $J$.\end{lem}

In particular Lemma~\ref{epred} shows that minimal $cl$-reductions exist. Following the idea in
Definition~\ref{defcore} we now define the $cl$-core.

\begin{dfn}{\rm  Let $R$ be a Noetherian ring and $cl$ a
closure defined on $R$.  The $cl$-core of an ideal $I$ is 
$cl$-$\core{I}= \bigcap \limits_{J \subset I} \ J$, where $J$ is a
$cl$-reduction of $I$.}
\end{dfn}

Recall, an ideal is basic if it does not have any nontrivial
reductions. We will say that an ideal is $cl$-basic if it does not
have any nontrivial $cl$-reductions. Clearly if $I$ is a basic ideal
$\core{I}=I$.  If $I$ is a $cl$-basic ideal then $cl$-$\core{I}=I$.
Note that we can restrict the intersection to the minimal
$cl$-reductions of $I$, when $R$ is a Noetherian local ring. In \cite{Va2} the second author has
discussed the partial ordering on the set of closure operations of a
ring defined as follows:  If $cl_1$ and $cl_2$ are closure
operations we say that $cl_1 \leq cl_2$ if and only if $I^{cl_1}
\subset I^{cl_2}$ for all ideals $I$ of $R$.

\begin{lem} \label{clineq}Let $cl_1$ be a closure operation and $cl_2$ be a Nakayama closure
operation
defined on a Noetherian local ring $R$ with $cl_1 \leq cl_2$.  Let $I$ be
an ideal.  If $J_1$ is a minimal $cl_1$-reduction of $I$ then there
exists a minimal $cl_2$-reduction $J_2$ of $I$ with $J_2 \subset
J_1$.
\end{lem}

\begin{proof}
Notice that $J_1 \subset I \subset J_1^{cl_1}$, as $J_1$ is a
$cl_1$-reduction of $I$.  Since $cl_1 \leq cl_2$ then $K^{cl_1}
\subset K^{cl_2}$ for all ideals $K \subset R$. Hence $J_1^{cl_1}
\subset J_1^{cl_2}$ and $J_1 \subset I \subset J_1^{cl_1} \subset
J_1^{cl_2}$. So $J_1$ is a $cl_2$-reduction of $I$ also. Now by
Lemma~\ref{epred}, there is a minimal $cl_2$-reduction of $I$
contained in $J_1$.
\end{proof}

One consequence of Lemma~\ref{clineq} is the following:

\begin{prop}\label{core inclu} Let $cl_1$ be a closure operation and $cl_2$ be a Nakayama closure operation
defined on a Noetherian local ring $R$ with $cl_1 \leq cl_2$.  Let $I$ be
an ideal. Then $cl_2$-$\core{I} \subset
cl_1$-$\core{I}$.
\end{prop}

\begin{proof}
We know that $cl_1$-$\core{I}=  \bigcap \limits_{J_1 \subset I} \
J_1$ where $J_1$ are $cl_1$-reductions of $I$. By Lemma~\ref{clineq}, for every
$cl_1$-reduction $J_1$ of $I$, there exists a minimal
$cl_2$-reduction, $J_2$ contained in $J_1$. Thus $cl_2 \textrm{-} \core{I} \subset  \bigcap \limits_{J_2 \subset J_1} \ J_2$ and $ \bigcap \limits_{J_2 \subset J_1} \ J_2 \subset  \bigcap \limits_{J_1 \subset I} \ J_1 =cl_1$-$\core{I}$.
\end{proof}

Let $R$ be a Noetherian ring of characteristic $p>0$. Note that $I^F \subseteq I^* \subseteq \overline{I}$ for all ideals
$I$ of $R$. The first inclusion is clear as $x \in I^F$ if $x^q \in
I^{[q]}$ for all $q>>0$ implies that $cx^q \in I^{[q]}$ for some $c
\in R^{o}$ namely, by taking $c=1$. The second inclusion holds by
\cite[Theorem~5.2]{HH}. In particular, we have the following
corollary regarding the Frobenius or $F$-core, the $*$-core and the
core, which is a $cl$-core where $cl$ is the integral closure.

\begin{cor} \label{contcore}
Let $R$ be an excellent local ring of characteristic $p>0$ and let $I$ be an ideal. Then
$\core{I} \subset {*}$-$\core{I} \subset
{F}$-$\core{I}$.
\end{cor}

Following Proposition in \cite[Proposition~17.8.9]{HS2} we see:

\begin{cor}
Let $R$ be a Noetherian local ring and let $I$ be an ideal. Then $\sqrt{I}=
\sqrt{cl{\textrm-}\core{I}}$ for any $cl \leq ^-$. In
particular, if $R$ is an excellent  local
ring of characteristic $p >0$ it follows that 
$\sqrt{I}=\sqrt{{*}{\textrm-}\core{I}}=\sqrt{
{F}{{\textrm-}\core{I}}}$.
\end{cor}

To better understand these minimal $cl$-reductions, Epstein mimicked
Vraciu's definition of $*$-independence in \cite{Vr1} to define
$cl$-independence.  The elements $x_1, \ldots, x_n$ are said to be
$cl$-independent if $x_i \notin (x_1, \ldots, \widehat{x_i}, \ldots,
x_n)^{cl}$, for all $1 \leq i \leq n$. Then he further refines the notion to that of strong
$cl$-independence.  An ideal is strongly $cl$-independent if every
minimal set of generators is $cl$-independent.  Epstein then showed
in \cite[Proposition~2.3]{Ep} that when $cl$ is a Nakayama closure, $J$ is a
minimal $cl$-reduction of $I$ if and only if $J$ is a strongly
$cl$-independent ideal.

In a Noetherian local ring $(R, \m)$ of characteristic $p>0$ Vraciu \cite{Vr1}
defined the special tight closure, $I^{*sp}$, to be the elements $x
\in R$ such that $x \in (\frak{m}I^{[q_0]})^*$ for some $q_0=p^{e_0}$ and some $e_0 \in \mathbb{N}$.
Vraciu shows in \cite[Proposition~4.2]{Vr1} that $I^{*sp}
\cap I=\frak{m}I$ if $I$ is generated by $*$-independent elements.
Note that the minimal $*$-reductions of $I$ are generated by
$*$-independent elements.  Epstein showed in \cite[Lemma~3.4]{Ep} that
$I^{*sp}=J^{*sp}$ for all $*$-reductions of $I$.

An ideal $I$ is said to have $cl$-spread, $\ell^{cl}(I)$, if all
minimal $cl$-reductions of $I$ have the same minimal number of generators.
As with the analytic spread, Epstein proves that $\mu (J) = \ell^{cl}(I)$ for all minimal $cl$-reductions $J$ of $I$. He also goes on to
prove \cite[Theorem~5.1]{Ep} that the $*$-spread is well defined
over an excellent analytically irreducible local domain of
characteristic $p >0$. Now if the $cl_1$- and the $cl_2$-spread are
defined for $I$, we have:

\begin{prop}Let $cl_1$ be a closure operation and $cl_2$ be a Nakayama closure operation
defined on a Noetherian local ring $R$ with $cl_1 \leq cl_2$.  Let $I$ be
an ideal with well-defined $cl_1$- and $cl_2$-spread. Then
$\ell^{cl_1}(I) \geq \ell^{cl_2}(I)$.
\end{prop}

\begin{proof}
Let $J_1$ be a $cl_1$-minimal reduction of $I$. Then $\mu(J_1)=
\ell^{cl_1}(I)$ \cite[Proposition~2.4]{Ep}. Also $J_1 \subset
I\subset J_1^{cl_{1}} \subset J_{1}^{cl_{2}}$, since $cl_1 \leq
cl_2$. Therefore $J_1$ is also a $cl_2$-reduction of $I$ (not
necessarily minimal). Hence $\mu(J_1) \geq \ell^{cl_2}(I)$ and
equality holds if and only if $J_1$ is a minimal $cl_2$-reduction of
$I$, according to \cite[Proposition~2.4]{Ep}. Hence
$\ell^{cl_1}(I)=\mu(J_1) \geq \ell^{cl_2}(I)$.
\end{proof}

In particular, we have the following corollary regarding the
Frobenius or $F$-spread, the $*$-spread and the analytic spread of an ideal:

\begin{cor}
Let $R$ be an excellent  local ring of
characteristic $p>0$  and let $I$ be an ideal. Then $\ell(I) \leq \ell^{*}(I)\leq
\ell^{F}(I)$.
\end{cor}

The analytic spread is bounded above by the dimension of the ring, but in
principle, the $cl$-spreads can grow arbitrarily large.  The
$cl$-spread of an ideal $I$ is however bounded by the minimal number
of generators of $I$, $\mu(I)$.

There are two invariants of a ring related to the analytic spread: the analytic
deviation and the second analytic deviation.  Recall that in a Noetherian local
ring, the {\it analytic deviation} of an ideal $I$ is $ad(I)=\ell(I)-{\rm
ht} \;I$.  Note that $I$ is equimultiple if $ad(I)=0$.  The {\it second
analytic deviation} of $I$ is $ad_2(I)=\mu(I)-\ell(I)$. We make the
following definitions with respect to the $cl$-spread of an ideal $I$.

\begin{dfn}{\rm
Let $R$ be a Noetherian local ring and $cl$ a closure operation on $R$.  Let $I$ be an ideal with a well defined $cl$-spread. The $cl$-deviation of  $I$ is
$cld(I)=\ell^{cl}(I)-{\rm ht} \;I$.  The second $cl$-deviation of $I$
is $cld_2(I)=\mu(I)-\ell^{cl}(I)$.}
\end{dfn}

\begin{rmrk} {\rm Let $R$ be a Noetherian local ring and $cl$ a closure operation on $R$. Let $I$ be an ideal with a well defined $cl$-spread. Assume $cl \leq ^-$.  The following are straightforward from the definition above.

\begin{enumerate}[ (a)]

\item \ Since $\ell(I) \leq \ell^{cl}(I)$ then $cld(I) \geq 0$.

\item \ Note that in a Cohen--Macaulay local ring, if $I$ is generated by
a system of parameters then $I$ is equimultiple and we have $cld(I)=0$.

\item \ Since $\ell(I) \leq \ell^{cl}(I)$, then $cld_2(I) \leq ad_2(I)$.
\end{enumerate}}
\end{rmrk}

Note if $I$ is $cl$-closed, then $\ell^{cl}(I)=\mu(I)$. If $I$ is a basic
ideal (i.e. $^{-}$-basic) and $cl \leq ^{-}$, then $\ell^{cl}(I)=\ell(I)$.
We would like to know how the core$(I)$ and the $cl$-$\core{I}$ are related
when $\ell(I)=\ell^{cl}(I)$.

\section{When $*$-core and core agree}

First we record some straightforward cases when the core and the
$*$-core agree.

\begin{prop} \label{1dim}  Let $R$ be a Noetherian local ring of characteristic
$p>0$ and $I$ be an ideal generated by a system of parameters. Then
$*$-$\core{I} =\core{I}$.
\end{prop}

\begin{proof}
An ideal generated by a system of parameters is basic and $*$-basic, hence the only
reduction (and $*$-reduction) of $I$ is $I$.  Thus $*$-$\core{I} = \core{I}=I$.
\end{proof}

Note, when $I$ is generated by a system of parameters, we may have
$I^* \subsetneq \overline{I}$, but the core and the $*$-core are
equal.

\begin{prop}\label{1-dim domain} Let $R$ be a one-dimensional local domain of characteristic
$p>0$ with infinite residue field and let $I$ be an ideal. Then $*$-$\core{I} = \core{I}$.
\end{prop}

\begin{proof}
If $I=0$ then the assertion is clear. Suppose then that $I \neq 0$ then $\ell(I)=1$.
By \cite[Example~1.6.2]{Hu2} it is known that for principal ideals $\overline{(x)}=(x)^*$ and also that
$I^*=(x)^*$, for some $x \in R$. Then
every minimal reduction and hence  minimal $*$-reduction of $I$ is principal.
Therefore we obtain equality of the core
and the $*$-core.
\end{proof}

We would like to show that in an excellent normal local ring the core
and the $*$-core agree for ideals of second $*$-deviation 1.  Note
that if $(R,\f{m})$ is  Gorenstein local isolated singularity of
characteristic $p>0$ with test ideal $\m$ and
$I$ is an ideal generated by a system of parameters, then
$*d_2(I)=1$ by Theorem~\ref{tcsop} since the tight closure is the
socle in this case, see also the proof of Corollary~\ref{sopgen}.

To show that the core and the $*$-core agree for ideals with $*d_2(I)=1$,
we will  begin by considering general minimal reductions. Recall:

\begin{dfn}{\rm (\cite{NR})
Let $R$ be a Noetherian local ring with infinite residue field $k$.
Let $I$ be an ideal generated by $f_1,\ldots,f_m$ and
let $t$ be a fixed positive integer.  We say that $b_1, \ldots,b_t$
are $t$ general elements in $I$ if there exists a dense open subset
$U$ of $\mathbb{A}_k^{tm}$ such that for $1 \leq j \leq m$, we have
$b_i=\sum\limits_{j=1}^{m} \lambda_{ij} f_j$, where
$\ul{\ul{\lambda}}=[\lambda_{ij}]_{ij} \in \mathbb{A}_R^{tm}$,
$\ol{\ul{\ul{\lambda}}} \in U$ vary in $U$ and
$\overline{\ul{\ul{\lambda}}}$ is the image of $\ul{\ul{\lambda}}$
in $\mathbb{A}_k^{tm}$. An ideal $J$ is called a general minimal
reduction of $I$ if $J$ is a minimal reduction of $I$ generated by $\ell(I)$
general elements.}
\end{dfn}

The next two Theorems show that general minimal $*$-reductions exist.

\begin{thm} \label{spreq}
Let $R$ be an excellent normal local ring of characteristic $p>0$ with infinite
perfect  residue field. Let $I$
be an ideal with  $*d_2(I)=1$. Then any ideal
generated by $\ell^{*}(I)$ general elements of $I$ is a minimal $*$-reduction of $I$.
\end{thm}

\begin{proof} Let $\ell^*(I)=s$. Then there exists $*$-independent elements $f_1,
\ldots, f_s \in I$ such that $I^*=(f_1, \ldots, f_s)^*$. Let
$J=(f_1, \ldots, f_s)$. Hence $J$ is a minimal $*$-reduction of $I$. By \cite[Lemma~2.2]{Ep} we know that this generating set of $J$ can be extended to a minimal generating set of
$I$. In other words, $I=(f_1,\cdots, f_s,f_{s+1})$. By
\cite[Theorem~2.1]{HV} we have that $J^*=J+J^{*sp}$. Also by
\cite[Lemma~3.4]{Ep} since $J \subset I$ and $J^*=I^*$ then
$J^{*sp}=I^{*sp}$. Therefore $I^*=J+I^{*sp}$ and $f_{s+1}$ can be
chosen such that $f_{s+1} \in I^{*sp}$.

Let $T=R[X_{ij}]$ where $1 \leq i \leq s$ and $1 \leq j \leq s+1$.
Let $a_{i}= \sum \limits_{j=1}^{s+1} X_{ij}f_j$ for $1 \leq
i \leq s$ and consider the $T$--ideal $\widetilde{J}=(a_1, \cdots,
a_s)$. Write $\ul{\ul{X}}$ for $[X_{ij}]_{ij}$.

Consider the $R$--homomorphism $\pi_{\ul{\ul{\lambda}}}:T
\rightarrow R$ that sends $\ul{\ul{X}}$ to $\ul{\ul{\lambda}}$,
where $\ul{\ul{\lambda}}\in \mathbb{A}_{R}^{s(s+1)}$. Notice that
for $\ul{\ul{\lambda_0}}=[\delta_{ij}]$ one has
$\pi_{\ul{\ul{\lambda_0}}}(\widetilde{J})=J$.

Let $\f{m}$ denote the maximal ideal of $R$ and $k=R/\f{m}$ the residue field of $R$. We need to find a
dense open set $U \subset
\mathbb{A}_k^{s(s+1)}$, such that
$\pi_{\ul{\ul{\lambda}}}(\widetilde{J})$ is also a $*$-reduction for all
$\overline{\ul{\ul{\lambda}}} \in U$.  Let $\overline{\lambda_{ij}}$
be the image of $\lambda_{ij}$ in $k$. Then the generators of
the $k$-vector space
$\pi_{\ul{\ul{\lambda}}}(\widetilde{J})/\f{m}\pi_{\ul{\ul{\lambda}}}(\widetilde{J})$
are $\overline{a_{i}}= \sum \limits_{j=1}^{s+1}
\overline{\lambda_{ij}}f_j$.

Define $L(\ul{\ul{\lambda}})=[\overline{\lambda_{ij}}]_{ij}$ to be the matrix defined by the
coefficients of the $\ol{a_i}$ for $i=1, \ldots, s$. Then $L(\ul{\ul{\lambda}})$ is a $s \times (s+1)$
matrix with coefficients in $k$. Suppose $L_s(\ul{\ul{\lambda}})$ is the $s \times s$
submatrix of $L(\ul{\ul{\lambda}})$ obtained by omitting the last column. We define the open
set $U \subset \mathbb{A}_k^{s(s+1)}$ to be set of $L(\ul{\ul{\lambda}})$'s satisfying
$\det(L_s(\ul{\ul{\lambda}})) \neq 0$. Since $\{L(\ul{\ul{\lambda}})  \mid\det(L_s(\ul{\ul{\lambda}}))=0\}$ is closed, then clearly $U$
is open.  Also $L(\ul{\ul{\lambda_0}}) \in U$ and thus $U$ is not empty. Therefore $U$ is an open dense set (see for example \cite[Lemma~2.9]{Kun}).

Since for any $\overline{\ul{\ul{\lambda}}} \in U$,
$\pi_{\ul{\ul{\lambda}}}(\widetilde{J})$ is a general reduction with
$\det(L_s) \neq 0$, then
$V=\pi_{\ul{\ul{\lambda}}}(\widetilde{J})/\f{m}\pi_{\ul{\ul{\lambda}}}(\widetilde{J})$
is an $s$-dimensional $k$-vector space with basis $\overline{a_1},
\ldots, \overline{a_s}$.  Row reducing $L(\ol{\ul{\ul{\lambda}}})$, we obtain the following
matrix:
$$\left(\begin{array}{cccccc} 1 & 0 & 0 & \cdots  & 0 & \beta_{1}' \\
0 & 1 & 0 & \cdots  & 0 & \beta_{2}' \\
0 & 0 & 1 & \cdots  & 0 &\beta_{3}' \\
\vdots & \vdots & \vdots &  \ddots & \vdots &\vdots \\
0 & 0 & 0 & \cdots  & 1 & \beta_{s}' \\
\end{array}\right)$$ where the $\beta_i' \in k$. This implies that an alternate basis for $V$ is
$\{\overline{f_1}+\beta_1' \overline{f_{s+1}}, \ldots,
\overline{f_s}+\beta_s'  \overline{f_{s+1}}\}$. Let $J_{\rm
gen}=(f_1+\beta_1 f_{s+1}, \ldots, f_s+\beta_s
f_{s+1})=\pi_{\ul{\ul{\lambda}}}(\widetilde{J})$, where $\beta_i$ is a preimage of $\beta_i'$ in $R$.

Case 1: Suppose that for all $1 \leq i \leq s$ we have that $\beta_i
\in \f{m}$. Let $K=J_{\rm gen}+\f{m}I$. Then we claim that
$K=J+\f{m}I$. To see this it is enough to check the inclusions for the generators of the ideals. Let  $\alpha$ be a generator of $K$. Then we can write
$\alpha=f_i+\beta_if_{s+1}+\delta$, where $\delta \in \f{m}I$. But
as $\beta_i \in \f{m}$ and $f_i \in J$ then $\alpha \in J+\f{m}I$.
Now let $\alpha'$ be a generator of $J+\f{m}I$. Then
$\alpha'=f_i+\delta'$, where $\delta' \in \f{m}I$. Since $\beta_i
\in \f{m}$ then $\delta'-\beta_i f_{s+1} \in \f{m}I$. Hence
$\alpha'=f_i+\beta_if_{s+1}+(\delta'-\beta_if_{s+1}) \in J_{\rm
gen}+\f{m}I=K$.

Next we claim that $(J+\f{m}I)^*=J^*$. Notice that $J \subset J+\f{m}I \subset I$.
Taking the tight closure we obtain $J^* \subset (J+\f{m}I)^* \subset I^*=J^*$.
Thus, $(J+\f{m}I)^*=J^*$. Overall we have the following inclusions:
\[
J_{\rm gen} \subset I \subset (J+\f{m}I)^*=(J_{\rm gen}+\f{m}I)^*
\]
Now by \cite[Proposition~2.1]{Ep} we have that $I \subset J_{\rm gen}^*$.

Case 2: Suppose that $\beta_i \not \in \f{m}$ for some $i$. Then without loss of generality we may assume that $i=s$ and $\beta_s=1$.
Then $J_{\rm gen}=(f_1+\beta_1f_{s+1}, \ldots, f_s+f_{s+1})$. Hence $f_1- \beta_1f_s \in J_{\rm gen}$.
Let $f_1'=f_1- \beta_1f_s$ and replace $f_1$ with $f_1'$. Continuing this way we may assume that
$J_{\rm gen}=(f_1, \ldots,f_{s-1}, f_s+f_{s+1})$. Suppose that
$f_{s+1} \not \in (f_1, \ldots,f_{s-1}, f_s+f_{s+1})^{*}$.

Since $f_{s+1} \in I \subset J^{*}$, then we may take $c \in R^{0}$
such that $cf_{s+1}^{q} \in J^{[q]}$ for every $q=p^e \gg 0$. Hence
$cf_{s+1}^{q}=\sum \limits_{i=1}^{s}r_{iq}f_i^q$. Then
$$cf_{s+1}^{q}+r_{sq}f_{s+1}^{q}=\sum \limits_{i=1}^{s}r_{iq}f_i^q+r_{sq}f_{s+1}^{q}=
\sum \limits_{i=1}^{s-1}r_{iq}f_i^q+r_{sq}(f_s^q+f_{s+1}^{q}).$$ Let
$c_q=c+r_{sq}$. Then
$c_qf_{s+1}^q=\sum \limits_{i=1}^{s-1}r_{iq}f_i^q+r_{sq}(f_s^q+f_{s+1}^{q})$
and in particular $c_qf_{s+1}^q \in (f_1, \ldots,
f_{s-1},f_s+f_{s+1})^{[q]}$. Therefore by Proposition~\ref{Ab}
there is a $q_0$, such that $c_{q} \in
\f{m}^{q/q_0}$ for all $q \geq q_0$. Also there is some $q_1$,
such that $c \notin \f{m}^{q_1}$. Hence for all $q \geq q_1$, we
have $r_{sq}=c_q-c \notin \f{m}^{q_1}$.

Notice that
$r_{sq}f_{s}^q=cf_{s+1}^{q}-\sum \limits_{i=1}^{s-1}r_{iq}f_i^q$. Since $r_{sq}\not \in \f{m}^{q_1}$  then $f_{s} \in (f_1, \ldots,
f_{s-1},f_{s+1})^{*}$  by
Proposition~\ref{Ab}.
Therefore $(f_1, \ldots, f_{s-1},f_{s+1})^{*}=I^*$ and thus $(f_1,
\ldots, f_{s-1},f_{s+1})$ is a minimal $*$-reduction of $I$.
However, since $f_{s+1} \in I^{*sp}$ then $(f_1, \ldots,
f_{s-1},f_{s+1})$ is not a minimal $*$-reduction of $I$, according
to \cite[Proposition~1.12(b)]{Vr2}, which is a contradiction.
Therefore $f_{s+1} \in J_{\rm gen}^{*}$ and thus $J_{\rm
gen}^{*}=I^{*}$.
\end{proof}

We are able to generalize Theorem~\ref{spreq} and relax the condition on $*d_2(I)$.

\begin{thm}\label{extension of main thm}
Let $R$ be an excellent normal local ring of characteristic $p>0$ with infinite perfect residue field. Let $I$
be an ideal. Then any ideal
generated by $\ell^{*}(I)$ general elements of $I$ is a minimal $*$-reduction of $I$.
\end{thm}

\begin{proof}
Let $\ell^*(I)=s$. Then there exist $*$-independent elements $f_1,
\ldots, f_s \in I$ such that $I^*=(f_1, \ldots, f_s)^*$. Let
$J=(f_1, \ldots, f_s)$. Hence $J$ is a minimal $*$-reduction of $I$.
By \cite[Lemma~2.2]{Ep} we know that any generating set of $J$ can
be extended to a minimal generating set of $I$. In other words, there exist $f_{s+1}, \ldots, f_{s+n} \in I$ such that
$I=(f_1,\cdots, f_s,f_{s+1},  \ldots, f_{s+n})$, where
$n=*d_2(I)$. By \cite[Theorem~2.1]{HV} we have that $J^*=J+J^{*sp}$.
Also by \cite[Lemma~3.4]{Ep} since $J \subset I$ and $J^*=I^*$ then
$J^{*sp}=I^{*sp}$. Therefore $I^*=J+I^{*sp}$ and thus $f_{s+1}, \ldots, f_{s+n}$ can be chosen such that $f_{s+1}, 
\ldots, f_{s+n} \in I^{*sp}$.

Let $\f{m}$
denote the maximal ideal of $R$ and $k=R/\f{m}$ be the residue field of $R$. As above we form an ideal generated by general elements and we may assume that
$$J_{\rm gen}=(f_1+\beta_{11} f_{s+1}+ \ldots +\beta_{1n}f_{s+n}, \ldots, f_s+\beta_{s1}f_{s+1}+\ldots+\beta_{sn}f_{s+n}),$$ where $\beta_{ij} \in R$. 

Case 1: Suppose that for all $1 \leq i \leq s$ and for all $1 \leq j
\leq n$ we have that $\beta_{ij} \in \f{m}$. Let $K=J_{\rm
gen}+\f{m}I$. Then we claim that $K=J+\f{m}I$. Let  $\alpha$ be a generator of $K$. The we can write $\alpha=f_i+\beta_{i1} f_{s+1}+ \ldots
+\beta_{in}f_{s+n}+\delta$, where $\delta \in \f{m}I$. But as
$\beta_{ij} \in \f{m}$ and $f_i \in J$ then $\alpha \in J+\f{m}I$.
Now let $\alpha'$ be a generator of $J+\f{m}I$. Then
$\alpha'=f_i+\delta'$, where $\delta' \in \f{m}I$. Since $\beta_{ij}
\in \f{m}$ for all $1 \leq i \leq s$ and for all $1 \leq j \leq n$
then $\delta'-(\beta_{i1} f_{s+1}+ \ldots +\beta_{in}f_{s+n}) \in
\f{m}I$. Hence $\alpha'=f_i+\beta_{i1} f_{s+1}+ \ldots
+\beta_{in}f_{s+n}+(\delta'- (\beta_{i1} f_{s+1}+ \ldots
+\beta_{in}f_{s+n})) \in J_{\rm gen}+\f{m}I=K$.

Next we claim that $(J+\f{m}I)^*=J^*$. Notice that $J \subset J+\f{m}I \subset I$.
Taking the tight closure we obtain $J^* \subset (J+\f{m}I)^* \subset I^*=J^*$.
Thus $(J+\f{m}I)^*=J^*$. Overall we have the following inclusions:
\[
J_{\rm gen} \subset I \subset (J+\f{m}I)^*=(J_{\rm gen}+\f{m}I)^*.
\]
Now by \cite[Proposition~2.1]{Ep} we have that $I \subset J_{\rm gen}^*$.

Case 2: Suppose $\beta_{ij} \not \in \f{m}$ for some $1 \leq i \leq s$ and $1 \leq j \leq n$.
Without loss of generality we may assume that $i=s$ and $j=n$ and that $\beta_{sn}=1$.
Hence $$J_{\rm gen}=(f_1+\beta_{11} f_{s+1}+ \ldots +\beta_{1n}f_{s+n}, \ldots, f_s+\beta_{s1}f_{s+1}+\ldots+f_{s+n}).$$

We claim that $I^{*}=J_{gen}^{*}$. We will proceed by induction on $n=*d_2(I)=\mu(I)-\ell^*(I)$. If
$n=0$ there is nothing to show and if $n=1$ then Theorem~\ref{spreq}
gives the result. So we assume that $n>1$ and the result holds for
any ideal $I'$ with $*d_2(I')=n-1$.

Let $g_i=f_i+\beta_{i1} f_{s+1}+ \ldots +\beta_{in}f_{s+n}$. Then
$$g_i-\beta_{in}g_s=(f_i-\beta_{in}f_s)+
\sum \limits_{j=1}^{n-1}(\beta_{ij}-\beta_{in}\beta_{sj})f_{s+j} \in J_{\rm
gen}.$$ Notice that $f_i'=f_i-\beta_{in}f_s \in J$ and let
$\beta_{ij}'=\beta_{ij}-\beta_{in}\beta_{sj}$. Therefore, we can
replace $f_i$ with $f_i'$ and $\beta_{ij}$ with $\beta_{ij}'$ to
assume that $J_{\rm gen}=(h_1, \ldots, h_{s-1}, h_s+f_{s+n})$, where $$h_i=f_i+\beta_{i1} f_{s+1}+ \ldots +\beta_{i(n-1)}f_{s+n-1}.$$  Let $L=(h_1, \ldots, h_s)$ and
$J_1=(f_1, \ldots, f_s, f_{s+1}, \ldots, f_{s+n-1})$.

Since $g_i$ is a general element for all $1 \leq i \leq s$ then
there exists $U \subset \mathbb{A}_{k}^{s(s+n)}$  a dense open set
such that the image of   $\beta_i=[0, \ldots, 0,1,0,\ldots,
\beta_{i1},\ldots, \beta_{i(n-1)},\beta_{in}]$ varies in $U$.
Consider the natural projection $\pi: \mathbb{A}_{k}^{s(s+n)}
\rightarrow \mathbb{A}_{k}^{s(s+n-1)}$ such that $\pi((\ul{a_{1}},
\ldots, \ul{a_{s+n-1}},\ul{a_{s+n}}))=(\ul{a_{1}}, \ldots,
\ul{a_{s+n-1}})$ for $\ul{a_{i}} \in \mathbb{A}_{k}^{s}$. Let
$W=\pi(U)$. As $U$ is dense and open then $U \neq \emptyset$ and
thus $W \neq \emptyset$ and $W$ is also open, since $\pi$ is an open map. Therefore $W$ is a dense open set. As $\beta_i$ is
allowed to vary in $U$ then $\pi(\beta_i)$ varies in $W$ and thus
$h_i$ is also a general element of $J_1$.

Notice that $J \subset J_1 \subset I$ and thus $J^*=J_1^*$. Hence
$\ell^*(J_1) = s$. Therefore $*d_2(J_1) \leq n-1$ and  by our
inductive hypothesis $L$ is a minimal $*$-reduction of $J_1$ and thus of $I$. Hence
$L^*=J_1^*=J^*=I^*$. We are claiming that $J_{\rm gen}^*=L^*=I^*$.
It is enough to show that $f_{s+n} \in J_{\rm gen}^*$. Suppose that
$f_{s+n} \not \in J_{\rm gen}^*$. Then as in the proof of
Theorem~\ref{spreq} we obtain that $h_s \in (h_1, \ldots,
h_{s-1},f_{s+n})^*$, which implies that $L^*=(h_1, \ldots, h_{s-1},f_{s+n})^*=I^*$. However, since $f_{s+n} \in I^{*sp}$ then
$(h_1, \ldots, h_{s-1},f_{s+n})$ is not a minimal $*$-reduction of
$I$, by \cite[Proposition~1.12(b)]{Vr2}, which is a contradiction.
Hence $f_{s+n} \in J_{\rm gen}^{*}$ and $J_{\rm gen}^{*}=L^{*}=I^*$.
\end{proof}

\begin{cor} \label{sopgen}
Let $(R,\f{m})$ be a Gorenstein normal local isolated singularity of
characteristic $p>0$ with infinite perfect  residue field. Suppose
that the test ideal of $R$ is $\f{m}$. Let $I=J^*$ where $J$ is a
parameter ideal minimally generated by $s$ elements. Then any ideal
generated by $s$ general elements of $I$ is a minimal $*$-reduction
of $I$.
\end{cor}

\begin{proof}
Suppose $J=(f_1, \ldots, f_s)$.  Then $J$ is a minimal $*$-reduction
of $I=J^*=(J:\m)$, where the last equality is obtained by
Theorem~\ref{tcsop}. Since $R$ is Gorenstein then the socle
$(J:\f{m})/J$ is a one dimensional vector space. Hence
$I=(f_1,\ldots, f_s,f_{s+1})$, where $f_{s+1} \not \in J$. Therefore
$\mu(I)=s+1$ and $*d_2(I)=1$. Thus by Theorem~\ref{spreq}, any ideal
generated by $s$ general elements is a minimal $*$-reduction of $I$.
\end{proof}

\begin{cor} \label{iso sing cores equal}
Let $(R,\f{m})$ be a Gorenstein normal  local isolated
singularity of characteristic $p>0$ with infinite perfect residue
field. Suppose that the test ideal of $R$ is $\f{m}$. Let $I=J^*$
where $J$ is generated by a system of parameters. Then
$\core{I}=*$-$\core{I}$.
\end{cor}

\begin{proof}
Since $I$ is $\f{m}$-primary then $ \ell(I)=\ell^*(I)=d$, where $d=\dim R$. By Corollary~\ref{sopgen} any ideal generated
by $d$ general elements is a general minimal $*$-reduction. Notice that these general minimal $*$-reductions
are also general minimal reductions of $I$, since $\ell(I)=d$.

Also, since $I$ is $\f{m}$-primary then by Theorem~\ref{CPU gen red} (\cite[Theorem~4.5]{CPU1}) we have that
$\core{I}$ is a finite intersection of general minimal reductions. Since each general minimal reduction is
also a minimal $*$-reduction then $*$-$\core{I} \subset \core{I} $. On the other hand
$\core{I} \subset *$-$\core{I}$, by Corollary~\ref{contcore}.
\end{proof}

\begin{thm}  \label{core agree} Let $R$ be a Cohen--Macaulay normal local domain of characteristic $p>0$
with infinite perfect residue field. Let $I$
be an ideal with $\ell^*(I)=\ell(I)=s$. We further assume that $I$ satisfies $G_s$ and is weakly
$(s-1)$-residually $S_2$. Then $ \core{I}=*$-$\core{I}$.
\end{thm}

\begin{proof}
We know that $\core{I} \subset *$-$\core{I}$ by
Corollary~\ref{contcore}. According to Theorem~\ref{CPU gen red} the core is a
finite intersection of general minimal reductions.  Since every general
minimal reduction is a minimal $*$-reduction by Theorem~\ref{extension of main thm}, we obtain the opposite inclusion.
\end{proof}

\section{The $*$-core in complete one dimensional
semigroup rings}

In Proposition~\ref{1dim}, we saw that the core and the $*$-core agree for
all ideals in a one dimensional domain of
characteristic $p>0$ with infinite residue field. In Huneke and Swanson's paper \cite{HS2}, one of the
first questions that they ask is: if $I$ is integrally closed, is
$\core{I}$ integrally closed? They settle this question in the setting of a
two-dimensional regular local ring.  Corso, Polini and Ulrich in \cite[Theorem~2.11]{CPU2}
showed that if $R$ is a Cohen--Macaulay
normal  local ring with infinite residue field then $\core{I}$ is integrally closed, when $I$ is a normal ideal of positive
height, universally weakly $(\ell -1)$-residually
$S_2$ and satisfies $G_{\ell}$ and $AN^{-}_{\ell-1}$, where $\ell=\ell(I)$.  A related question is: if $I$ is
tightly closed, is $*$-$\core{I}$ tightly closed?  We will consider this
question now for complete one-dimensional semigroup rings with test ideal
equal to the maximal ideal.  The second author showed the
following:

\begin{thm} $($\cite{Va1}$)$ Let $R$ be a one-dimensional domain.  The test
ideal of $R$ is equal to the conductor, i.e. $\tau=\frak{c}=\{c \in
R \mid \phi(1)=c,\phi \in {\rm{Hom}}_R(\overline{R},R)\}$, where $\overline{R}$ denotes the integral closure of $R$.
\end{thm}

Note that in a one-dimensional local semigroup ring, the semigroup is a
sub-semigroup of $\mathbb{N}_0$.  For each sub-semigroup $S$ of $\mathbb{N}_0$, there is a
smallest $m$ such that for all $i \geq m$, $i \in S$. The conductor of
such a one dimensional semigroup ring is $\frak{c}=<t^m, t^{m+1},t^{m+2},
\ldots>$, \cite[Exercise~21.11]{Ei}.  Hence the test ideal in a
one-dimensional semigroup ring is the maximal ideal, if the conductor is
the maximal ideal.  This can only happen if the semigroup has the form
$\{n+i \mid i \geq 0\}$ for some $n\geq 0$.  Hence if $R$ is complete the
ring is of the form $R=k[[t^n,t^{n+1}, \ldots, t^{2n-1}]]$, where $k$ is a field.  As in
\cite[Proposition~4.1]{Va2}, we will show that the principal ideals of $R$ are of
a given form:

\begin{prop}  \label{principal ideals} Let $R=k[[t^n,t^{n+1},\ldots,t^{2n-1}]]$ be a one-dimensional local semigroup ring and $k$ a field. Each nonzero nonunit principal ideal of $R$ can be expressed in the form
$$(t^m+a_1t^{m+1}+\cdots+a_{n-1}t^{m+n-1}),$$ where $a_i \in k$ and $m \geq
n$.\end{prop}

\begin{proof}  Suppose $0 \neq f \in R$.  Thus, after multiplying by a
nonzero element of $k$ we may assume that  $f=t^m+a_1t^{m+1}+a_2t^{m+2}+\cdots$ for some $a_{ij} \in k$ and for some  $m \geq n$. We will show that $t^{r} \in (f)$ for $r \geq m+n$. Hence
$t^m+a_1t^{m+1}+\cdots+a_{n-1}t^{m+n-1} \in (f)$. Similarly, for $r \geq m+n$ we obtain $t^{r}
\in (t^m+a_1t^{m+1}+\cdots+a_{n-1}t^{m+n-1})$.
Hence $f \in (t^m+a_1t^{m+1}+\cdots+a_{n-1}t^{m+n-1})$.

Let $g \in k[[t]]$. Note that $t^{r-m}g \in
R$. Therefore if $g$ is a unit in
$k[[t]]$, then $t^{r-m}g^{-1} \in R$
also. In $k[[t]]$ we have $f=t^{m}(1+a_1t+a_2t^2+\cdots)=t^{m}g$, for some unit $g \in k[[t]]$.  Also notice
that $t^{r-m}g^{-1}f=t^r$.

Similarly $t^r \in
(t^m+a_1t^{m+1}+\cdots+a_{n-1}t^{m+n-1})$. Since
$$f-(t^m+a_1t^{m+1}+\cdots+a_{n-1}t^{m+n-1})=a_nt^{2n}+a_{n+1}t^{2n+1}+
\cdots$$ is an element of $(f) \bigcap
(t^m+a_1t^{m+1}+\cdots+a_{n-1}t^{m+n-1})$, it follows  that
$$(t^m+a_1t^{m+1}+\cdots+a_{n-1}t^{m+n-1})=(f).$$ Therefore all principal
ideals of $R$ are of the form
$(t^m+a_1t^{m+1}+\cdots+a_{n-1}t^{m+n-1})$.
\end{proof}

\begin{prop}\label{tightly closed form}  Let $R=k[[t^n,t^{n+1},\ldots,t^{2n-1}]]$ be a one-dimensional local semigroup ring and $k$ be an infinite field of characteristic $p>0$.
Any tightly closed ideal in $R$ is of the form
$(t^m,t^{m+1},\ldots,t^{m+n-1})$ for some $m \geq n$.
\end{prop}

\proof Suppose $I$ is a tightly closed ideal in $R$.  Since $R$ is a one-dimensional
local domain, there is a principal ideal $(f) \in I$, with $(f)^*=I$.  By
Proposition~\ref{principal ideals},
$(f)=(t^m+a_1t^{m+1}+\cdots+a_{n-1}t^{m+n-1})$, for some $m \geq n$ and $a_i \in k$.
According to \cite[Example 1.6.2]{Hu2},
$I^*=\overline{I}=\overline{(x)}$ for some $x \in I$. Moreover,
by \cite[Theorem 3.8]{Va1} it follows that $\overline{(x)}=(x)\overline{R}\ \cap R$. Hence $$(f)^*=\overline{(f)}=(f)k[[t]] \cap
R=(t^m,t^{m+1},\ldots,t^{m+n-1}). \mathproofbox$$

\begin{prop}\label{core tightly closed}
Let $R=k[[t^n,t^{n+1},\ldots,t^{2n-1}]]$ be a one-dimensional local semigroup ring and $k$ be an infinite field of characteristic $p>0$. Let $I $ be a tightly closed ideal.  Then
$*$-$\core{I}$ is also tightly closed.
\end{prop}

\begin{proof}
If $I=(0)$, then clearly $*$-$\core{I}=(0)$ and thus the assertion is clear.
Since $R$ is a one-dimensional domain then $\core{I}=*$-$\core{I}$ by Proposition~\ref{1-dim domain}. If $I$
is basic then $I$ is also $*$-basic and again the assertion is clear. So suppose
$I$ is not basic, nonzero and
tightly closed. Then $I=(t^m,t^{m+1},\ldots, t^{m+n-1})$ for some $m \geq n$,
by Proposition~\ref{tightly closed form}. Since $I$ is non-zero then $I$ is
$\f{m}$-primary, where $\f{m}$ is the maximal ideal of $R$. Hence by Theorem~\ref{CPU gen red}
we have that $\core{I}=\bigcap \limits_{i=1}^{s}(f_i)$,
for some positive integer $s$ and $(f_i)$ general minimal reductions of $I$ for all
$1 \leq i \leq s$. Let $(f_i)$ be such a general minimal reduction. Then
$(f_i)=(t^m+a_{i1}t^{m+1}+\cdots+a_{i(n-1)}t^{m+n-1})$ for some $a_{ij}
\in k$, since $f_i$ is a general element in $I$. As in the proof of
Proposition~\ref{principal ideals}, we see that $$t^r \in
(t^m+a_1t^{m+1}+\cdots+a_{n-1}t^{m+n-1}),$$ for all $r \geq m+n$.  Hence
$(t^{m+n},t^{m+n+1}, \ldots, t^{m+2n-1}) \subset (f_i)$ for all $i$ and
thus $$(t^{m+n},t^{m+n+1}, \ldots, t^{m+2n-1}) \subset *\textrm{-}\core{I}.$$

On the other hand let $g \in *$-$\core{I}$. Hence $g \in
\bigcap \limits_{i=1}^{s}(f_i)$. It is clear that
$(g) \neq (f_i)$ for some $i$. Then $g=a
(t^m+a_{i1}t^{m+1}+\cdots+a_{i(n-1)}t^{m+n-1})$ for some $a \in R$ and
$a_{ij} \in k$. If $a$ is a unit then $(g)=(f_i)$, which is a
contradiction. Hence we may assume that $a$ is not a unit. Thus $a=\beta_1
t^n+\beta_2 t^{n+1} + \ldots$ and $$g=\gamma_0 t^{m+n}+ 
\ldots +\gamma_{n-1} t^{m+2n-1}+t^n(\gamma_{n} t^{m+n}+
 \ldots +\gamma_{2n-1} t^{m+2n-1})+\ldots.$$ Therefore
$g \in (t^{m+n},t^{m+n+1}, \ldots, t^{m+2n-1})$ and thus $$*\textrm{-}\core{I}
\subset (t^{m+n},t^{m+n+1}, \ldots, t^{m+2n-1}).$$ Finally notice that
$*$-$\core{I}=(t^{m+n},t^{m+n+1}, \ldots, t^{m+2n-1}) $ is a tightly closed ideal.
\end{proof}

Note that by Proposition~\ref{1dim} in a one-dimensional domain  with infinite residue field we have $\core{I}= *\textrm{-}\core{I}$   and
the tight closure of an ideal agrees with the integral closure. Thus we obtain:

\begin{cor} \label{int closed}
Let $R=k[[t^n,t^{n+1},\ldots,t^{2n-1}]]$ be a one-dimensional local semigroup ring and $k$ be an infinite field of characteristic $p>0$. Let $I $ be an integrally closed ideal. Then $\core{I}$ is integrally closed.
\end{cor}

\medskip

\begin{rmrk}{\rm
As mentioned above the question of whether the core of an integrally closed ideal is also integrally closed was
first addressed by Huneke and Swanson \cite{HS1}. They answer this question positively when the
ring is a 2--dimensional regular ring \cite[Corollary~3.12]{HS1}. This question was also addressed
by several other authors later (see \cite[Theorem~2.11, Corollary~3.7]{CPU2}, \cite[Corollary~4.6]{PU},
and \cite[Proposition~5.5.3]{HyS1}).

We note here that Corollary~\ref{int closed} is not covered by any of the results mentioned above.
In \cite[Corollary~3.7]{CPU2} and \cite[Corollary~4.6]{PU} it is required that the ring $R$ is
Gorenstein. The ring $R=k[[t^n,t^{n+1},\ldots,t^{2n-1}]]$ with $k$ an infinite field of characteristic
$p>0$ is not Gorenstein unless $n=2$.
In \cite[Theorem~2.11]{CPU1} the Gorenstein condition can be relaxed to Cohen--Macaulay rings, but in
addition the Rees algebra of $I$ and $I$ are assumed to be normal and $J:I$ is independent of $J$ for
every minimal reduction $J$ of $I$. Notice that the ideal $I$ in Corollary~\ref{int closed} is normal
and $J:I=\tau=\f{m}$ is independent of the minimal reduction $J$ of $I$. However, the Rees algebra of $I$ is
not normal, since $R$ is not normal. Finally, in \cite[Proposition~5.5.3]{HyS1} it is assumed that the
ring $R$ is Cohen--Macaulay, $R$ contains the rational numbers and the Rees algebra of $I$ is
Cohen--Macaulay whereas the ring in Corollary~\ref{int closed} need not contain the rational numbers.}
\end{rmrk}

\section{Examples}

Since the tight closure of an ideal is much closer to the ideal than the
integral closure we expected to find examples of ideals $I$ where the
$ \core{I} \subsetneq *$-$\core{I}$.  The following example gives a family
of rings where $\core{\f{m}^2} \subsetneq  *\textrm{-}\core{\f{m}^2} $.

\begin{exm}\label{ex ell neq ell*}{\rm Let $R=\mathbb{Z}/p\mathbb{Z}(u,v,w)[[x,y,z]]/(ux^p+vy^p+wz^p)$, where $p$ is prime. Then $R$
is a normal domain \cite{Ep}.  In \cite{VV} Vraciu and the second
author computed the test ideal of $R$ to be $\f{m}^{p-1}$, where $\f{m}$ is the maximal ideal of $R$. Let $k=\mathbb{Z}/p\mathbb{Z}(u,v,w)$.
Notice that since $\ell(\m^2)=2$ then $\ell^{*}(\m^2)$ is either $2$ or $3$.  

We begin by showing that $\ell^{*}(\m^2)=3$, regardless of the characteristic $p$.  We claim that $J=(y^2,yz,z^2)$ is a minimal $*$-reduction of $\f{m}^2$.  To establish this we must show that  $y^2,yz,z^2$ are $*$-independent elements and that $J^*=\m^2$.

We note that $y^2,z^2$ is a system of parameters and therefore by Theorem~\ref{tcsop} we have $(y^2,z^2)^*=(y^2,z^2):\m^{p-1}$. Hence,
\begin{eqnarray*}
(y^2,z^2)^*=(y^2,z^2):\f{m}^{p-1}&=&(y^2,z^2,x^{p-1}yz):\f{m}^{p-2} \\
&=&(y^2,z^2,x^{p-1}y,x^{p-1}z,x^{p-2}yz):\f{m}^{p-3}=\cdots (y^2,z^2)+\f{m}^{3}.
\end{eqnarray*}

In particular, this shows that $yz \notin (y^2,z^2)^{*}$. It remains to establish that $z^2 \notin (y^2,yz)^{*}$ and $y^2 \notin (yz,z^2)^{*}$. Notice that  $(y^2,yz)^{*} \subset (y^2,yz):\m^{p-1}$, by the definition of the test ideal. As above,
$$(y^2,yz):\f{m}^{p-1}=(y^2,yz,x^{p-1}y):\f{m}^{p-2}=(y^2,z^2,x^{p-2}y):\f{m}^{p-3}=\cdots =(y^2,yz,xy).$$
One then observes that since $z^2 \notin (y^2,yz,xy)$ then $z^2 \notin (y^2,yz)^{*}$. Similarly $y^2 \notin (yz,z^2)^{*}$. Therefore $y^2,yz, z^2$ are $*$-independent elements,

Next we must show that $J^*=\m^2$. The calculations depend on the characteristic $p$ and thus we separate the computations. 

For $p=2$ notice that  $x^2=\frac{v}{u}y^2+\frac{w}{u}z^2 \in J$ and
$$(xz)^2=(\frac{v}{u}y^2+\frac{w}{u}z^2)z^2=\frac{v}{u}(yz)^2+\frac{w}{u}(z^2)^2.$$
Hence $xz \in J^F\subset J^*$.  Similarly $xy \in J^*$ and thus $\f{m}^2=J^*$. Therefore $J$ is indeed a minimal $*$-reduction of $\m^2$. 

For $p \geq 3$ notice that
$$(x^2)^p=(x^p)^2=(\frac{v}{u}y^p+\frac{w}{u}z^p)^2=
\frac{v^2}{u^2}(y^2)^p+2\frac{vw}{u^2}(yz)^p+\frac{w^2}{u^2}(z^2)^p\in
J^F$$ and $(xz)^p=(\frac{v}{u}y^p+\frac{w}{u}z^p)z^p$. Thus $xz \in
J^F\subset J^*$.  Similarly $xy \in J^*$ and thus $J^*=\f{m}^2$. Therefore $J$ is again a minimal $*$-reduction of $\m^2$ and thus $\ell^*(\m^2)=3$ for any characteristic.

Next we continue with the computations of $*$-$\core{\m^2}$ and $\core{\m^2}$. Once again these depend on the characteristic $p$ and thus we separate the computations.

For $p=2$, we compute the $*$-core of $\f{m}^2$ in the following manner: Recall that  $J=(y^2,yz,z^2)$ is a minimal $*$-reduction of $\f{m}^2$. In addition we note that $(x^2,xy,y^2)$, $ (x^2,xz,z^2)$, $(y^2,yz,z^2)$, and $(yz,xz,xy)$ are all minimal $*$-reductions of $\f{m}^2$. Hence
\begin{eqnarray*} *{\textrm -} \core{\f{m}^2} &\subset& (x^2,xy,y^2)\cap (x^2,xz,z^2) \cap
(y^2,yz,z^2)\cap (yz,xz,xy)\\
&=&(x^2,y^2,z^2,xyz) \cap (yz,xz,xy)=\f{m}^3. \end{eqnarray*}
Note that $\f{m}^3=\f{m}J^* \subset J$ for all $J$ minimal $*$-reductions of $\f{m}^2$, since $\f{m}$ is
the test ideal. Hence $*$-$\core{\f{m}^2}=\f{m}^3$.

For $p\geq 3$ we estimate  the $*$-core of $\f{m}^2$. Again,
$J=(y^2,yz,z^2)$ is a minimal $*$-reduction of $\f{m}^2$ and similarly $(x^2,xy,y^2)$, $(x^2,xz,z^2)$, $(y^2,yz,z^2)$, and
$(yz,xz,xy)$ are all minimal $*$-reductions of $\f{m}^2$.
  As the test ideal is
$\f{m}^{p-1}$, we see that $\f{m}^{p-1}J^*=\f{m}^{p+1} \subset J$
for all minimal $*$-reductions $J$ of $\f{m}^2$.  Therefore
\begin{eqnarray*}*{\textrm-} \core{\f{m}^2} &\subset& (x^2,xy,y^2)\cap (x^2,xz,z^2) \cap
(y^2,yz,z^2)\cap (yz,xz,xy)\\&=&m^{p+1}+(xyz,x^2y^2,x^2z^2,y^2z^2).\end{eqnarray*}
Hence  $\f{m}^{p+1} \subset*$-$\core{\f{m}^2} \subset \f{m}^{p+1}+(xyz,x^2y^2,x^2z^2,y^2z^2) $. 
We remark here that we have not been able to establish a closed formula for $*$-$\core{\m^2}$ for $p \geq 3$, but we will show that the above inclusions are enough to show that $\core{\m^2} \subsetneq *$-$\core{\m^2}$.

Last we compute $\core{\m^2}$. Recall that $\ell(\f{m}^2)=2$ and notice that  $H=(x^2,yz)$ is a
minimal reduction of $\f{m}^2$ in any characteristic.

For $p=2$ the reduction number of $\f{m}^2$  with respect to $H$ is $1$. Since ${\rm char} \;k=2>1$ then we may
use the formula for the core as in \cite[Theorem~4.5]{PU}. Hence
$\core{\f{m}^2}=H^2:\f{m}^2=\f{m}^4$, where the last
equality follows from calculations using the computer algebra program Macaulay~2 \cite{M2}. Therefore
$\m^4=\core{\f{m}^2} \subsetneq *$-$\core{\f{m}^2}=\m^3$.

For $p=3$ the reduction number of $\f{m}^2$ with respect to $H$
is $2$. Since now ${\rm char} \;k=3>2$ we may again use the formula
as in \cite[Theorem~4.5]{PU}. Thus
$\core{\f{m}^2}=H^3:\f{m}^4=\f{m}^5$, where the last equality is again obtained using the computer algebra program
Macaulay~2 \cite{M2}. Notice that since $\f{m}^{4} \subset
*$-$\core{\f{m}^2}$ then $\core{\f{m}^2}
\subsetneq *$-$\core{\f{m}^2}$ again. }
\end{exm}

When the  analytic spread and the $*$-spread agree, it is not necessarily the case
that all reductions of an ideal are $*$-reductions.  However, the
following example exhibits that even so, the core and the $*$-core agree
for the maximal ideal in the following ring.  In some sense, the following
example prompted us to prove Theorem~\ref{spreq}, Theorem~\ref{extension
of main thm} and Theorem~\ref{core agree}.

\begin{exm}
{\rm Let $R=k[[x,y,z]]/(x^2-y^3-z^7)$, where the $k$ is an infinite
field and ${\rm char}\;k >7$. Let $\f{m}=(x,y,z)$ denote the maximal ideal of $R$.  We
observe first that $\f{m}$ is the test ideal, \cite{Va1}.

We will show that $*$-spread of $\f{m}$ is 2, $\ell(\f{m})=2$ and
$\core{\f{m}}=\f{m}^2=*$-$\core{\f{m}}$.

First note that $R$ is a $2$--dimensional Gorenstein local ring and hence
$\ell(\f{m})=2$. Let $J=(y,z)$. Then $J$ is a minimal reduction of $\f{m}$
with reduction number $1$. Since ${\rm char}\; k >1$ then
$\core{\f{m}}=J^2:\f{m}=\f{m}^{2}$ by \cite[Theorem~4.5]{PU}. Notice
that this does not agree with the formula in Hyry-Smith \cite[Theorem~4.1]{HyS2}
or in Fouli-Polini-Ulrich \cite[Theorem~4.4]{FPU} since $a=42-21-14-6=1$ and
$\core{\f{m}} \neq \f{m}^{2+a+1}=\f{m}^{4}$.  The hypothesis that
$\f{m}$ is generated by elements of degree 1 is important in their
formula.

On the other hand, $J$ is also a minimal $*$-reduction of $\f{m}$.   Note that
$y,z$ form a system of parameters and  by Theorem~\ref{tcsop} we have that 
$(y,z)^*=(y,z):\f{m}=(x,y,z)=\f{m}$. Therefore $\ell^*(\f{m})=2=\ell(\f{m})$. We claim that $J_1=(x+z,y)$ and
$J_2=(x+y,z)$ are also minimal $*$-reductions. Denote
$$p_n(x,y)=x^{n}+x^{n-1}y+\ldots +xy^{n-1}+y^n.$$ Note that if $n$ is
odd, $x^n+y^n=(x+y)p_{n-1}(x,-y)$.  Now we can see that
$$(x+z)p_6(x,-z)+y^3=x^7+z^7+y^3=x^7+x^2-y^3+y^3=x^2(1+x^5).$$ Since
$(1+x^5)$ is a unit in $R$, then $x^2 \in (x+z,y)$.  Since
$x(x+z)=x^2+xz$ we also observe that $xz \in (x+z,y)$ and similarly,
we see that $z^2 \in (x+z,y)$. Hence $\f{m}^2 \subset (x+z,y)$ and
thus $\f{m} \subset (x+z,y):\f{m}=(x+z,y)^* \subset \f{m}$, i.e.
$(x+z,y)^*=\f{m}$.  Using the same argument exchanging $y$ and $z$
and exchanging the powers 3 and 7, we see that indeed $J_2$ is a minimal
$*$-reduction of $\f{m}$.

Let $K$ be a minimal $*$-reduction of $\f{m}$.  Then $\f{m} = K^*$ and $\f{m}^2=\f{m}K^{*} \subset K$.
Therefore $\f{m}^2 \subset *$-$\core{\f{m}}$.  We can easily see that
$\f{m}^2=J \bigcap J_1 \bigcap J_2$ and hence conclude that
$\f{m}^2$ is in fact  $*$-$\core{\f{m}}$ and $\core{\f{m}}=*$-$\core{\f{m}}$.}

\end{exm}

\begin{acknowledgements}
We would like to thank Craig Huneke for advice pertaining to the proof of
Theorem~\ref{spreq} and for hosting the first author at the University of
Kansas while some of the work was completed.  We would also like to thank
Bernd Ulrich for reading the manuscript and for pointing out to us that
the result of Corollary~\ref{int closed} was not known before.  We also
thank Adela Vraciu for suggesting Lemma~\ref{clineq}. Finally, we thank the referee for the careful reading of the article and the comments that improved the exposition.
\end{acknowledgements}

\end{document}